\documentclass{amsart}

\usepackage[dvipsnames]{xcolor}
\usepackage{amsmath}
\usepackage{enumerate}
\usepackage{enumitem}
\usepackage{amssymb}
\usepackage{extarrows}
\usepackage{amsthm}
\usepackage{thmtools}
\usepackage{mathrsfs}
\usepackage{tikz-cd}
\usepackage[colorlinks, linkcolor = SeaGreen, citecolor = SeaGreen, urlcolor = SeaGreen]{hyperref}

\hypersetup{colorlinks = true}

\newtheorem{thm}{Theorem}[section]

\newtheorem{lem}[thm]{Lemma}
\newtheorem{coro}[thm]{Corollary}
\newtheorem{prop}[thm]{Proposition}
\theoremstyle{definition}

\theoremstyle{remark}
\newtheorem{remark}[thm]{Remark}
\newtheorem{ques}[thm]{Question}

\newcommand{\ppi}{\pi}
\newcommand{\ii}{\mathrm{i}}
\newcommand{\ee}{\mathrm{e}}
\newcommand{\dd}{\,\mathrm{d}}
\newcommand{\rd}{\mathrm{d}}

\newcommand{\CC}{\mathbb{C}}
\newcommand{\DD}{\mathbb{D}}

\newcommand{\QQ}{\mathbb{Q}}
\newcommand{\RR}{\mathbb{R}}
\newcommand{\TT}{\mathbb{T}}

\newcommand{\NN}{\mathbb{N}}

\newcommand{\Hol}{\mathrm{Hol}}

\newcommand{\orb}{\mathrm{orb}}
\newcommand{\Per}{\mathrm{Per}}

\newcommand{\wind}{\mathrm{wind}}

\DeclareMathOperator{\re}{Re}
\DeclareMathOperator{\dist}{dist}

\DeclareMathOperator{\ind}{index}

\DeclareMathOperator{\spa}{span}

\title[Hypercyclic Bergman--Toeplitz Operators]{Hypercyclic Bergman--Toeplitz operators with some harmonic symbols}

\author[Qianrui Leng]{Qianrui Leng}
\address{\textsuperscript{1} College of Mathematics and Statistics, Chongqing University, Chongqing, 401331, P. R. China}
\email{17788619117@163.com}

\author[Xiaoyan Zhang]{Xiaoyan Zhang}
\address{\textsuperscript{2} School of Sciences, Southwest Petroleum University, Chengdu, 610500, P. R. China}
\email{15282233835@163.com}

\author[Xianfeng Zhao]{Xianfeng Zhao}
\address{
\textsuperscript{3}
College of Mathematics and Statistics, Chongqing University, Chongqing, 401331, P. R. China}
\email{xianfengzhao@cqu.edu.cn}



\begin{document}
	\begin{abstract}
		Little is known about the dynamics of Toeplitz operators on the Bergman space, and even in the case of harmonic symbols, few results are available. In this paper, we  study the dynamical properties of Toeplitz operators with symbols of the form $a \overline{z} + p (z)$ on the Bergman space, where $a\neq 0$ and $p$ is analytic on the closed unit disk. We obtain some necessary conditions and some sufficient conditions for such operators to be hypercyclic, weakly mixing, mixing, chaotic, or frequently hypercyclic. As an application, we obtain a complete characterization of hypercyclicity for Toeplitz operators with harmonic linear polynomial symbols. We also show that Toeplitz operators with the same symbol may exhibit different hypercyclic behavior on the Hardy and Bergman spaces.
	\end{abstract}
	
	\maketitle
	
	\section{Introduction}\label{S1}
	
In this paper, we use $\NN$, $\NN_+$, $\QQ$, $\RR$ and $\CC$ to denote the sets of nonnegative integers, positive integers, rationals, real numbers and complex numbers, respectively. Besides, we denote $\DD = \{ z \in \CC : |z| < 1 \}$, $\overline{\mathbb D}=\{ z \in \CC : |z| \leqslant 1 \}$, $\TT = \{ z \in \CC : |z| = 1 \}$, $\widehat{\DD} = \{ z \in \CC : |z| > 1 \}$, $\DD_0 = \DD \backslash \{ 0 \}$, and $\overline{\DD}_0 = \overline{\DD} \backslash \{ 0 \}$. The open disk in the complex plane centered at $z$ with radius $r$ is denoted by $D (z ,r)$.
	
	The classical \textit{Hardy space} is defined as
	\begin{align*}
		H^ 2 = \left\{ f \in L^ 2 (\TT, \rd m) : \int_{\TT} f (\ee^ {\ii t}) \ee^ {\ii n t} \dd m(\ee^ {\ii t}) = 0 \textrm{ for } n \in \NN_+ \right\},
	\end{align*}
	where $\rd m$ is the normalized Lebesgue measure on $\TT$. It can be shown that $H^ 2$ is a closed subspace of $L^ 2 (\TT, \rd m)$, and we use $\mathbf{P}$ to denote the orthogonal projection from $L^ 2 (\TT, \rd m)$ onto $H^ 2$. For an essentially bounded function $\varphi$ on $\TT$ (with respect to $\rd m$), the \textit{Toeplitz operator with symbol $\varphi$} on the Hardy space, or simply the Hardy--Toeplitz operator with symbol $\varphi$, is defined as
	\begin{align*}
		\mathbf{T}_{\varphi} f = \mathbf{P} (\varphi f), \qquad f \in H^ 2,
	\end{align*}
	which is a bounded linear operator on $H^ 2$. The following three identities of Hardy--Toeplitz operators are elementary:
	\begin{align} \label{eq4}
		\mathbf{T}_{\varphi} + \mathbf{T}_{\psi} = \mathbf{T}_{\varphi + \psi}, \qquad \lambda \mathbf{T}_{\varphi} = \mathbf{T}_{\lambda \varphi}, \qquad \textrm{and} \qquad \mathbf{T}_{\varphi}^ * = \mathbf{T}_{\overline{\varphi}},
	\end{align}
	where $\varphi, \psi \in L^ {\infty} (\TT, \rd m)$ and $\lambda \in \CC$. It is also clear from the definition that
	\begin{align}
		\| \mathbf{T}_{\varphi} \| \leqslant \| \varphi \|_{L^ {\infty} (\TT, \rd m)}, \qquad \varphi \in L^ {\infty} (\TT, \rd m). \label{ineq3}
	\end{align}
	For the theory of Hardy spaces and Toeplitz operators thereon, one can consult \cite{Dou, Zhu}.
	
	The study of hypercyclic Toeplitz operators goes back to Rolewicz \cite{Rol}, who proved that the Toeplitz operator $\mathbf{T}_{\alpha \overline{z}}$ on the Hardy space is hypercyclic if $|\alpha| > 1$. For the dynamical properties of general weighted shifts, we refer the reader to \cite{Sal1, Sal2, G-E}. A much more general result was obtained by Godefroy and Shapiro \cite{GS}. Indeed, they showed that in many analytic function Hilbert spaces on a domain $\Omega$, including the Hardy space and the Bergman space to be introduced later, the Toeplitz operator with non-constant bounded co-analytic symbol $\overline{q}$ is hypercyclic if and only if $q (\Omega)$ intersects the unit circle $\TT$. After about 20 years, Shkarin \cite{Shk} and Baranov and Lishanskii \cite{BL} gave a complete characterization of the hypercyclicity of the Toeplitz operator $\mathbf{T}_{a \overline{z} + c_0 + c_1 z}$.
	
	\begin{thm} \label{theorem: Shkarin's result}
		Suppose that $\varphi (z) = a \overline{z} + c_0 + c_1 z$, where $a, c_0, c_1 \in \CC$ and $a \neq 0$. Then $\mathbf{T}_{\varphi}$ is hypercyclic on the Hardy space if and only if $|a| > |c_1|$ and $\varphi (\DD) \cap \TT \neq \varnothing$.
	\end{thm}

	Baranov and Lishanskii also investigated general harmonic polynomial symbols, and in particular obtained the following result:
	
	\begin{thm}[\cite{BL}] \label{theorem: hypercyclic Hardy-Toeplitz operators}
		Let  $a \in \CC \backslash \{ 0 \}$, $p$  be a bounded analytic function on $\DD$, and $\varphi (z) = a \overline{z} + p(z)$. Let $\widetilde{\varphi} (z) = \frac{a}{z} + p (z)$.
		\begin{enumerate}[label = \textup{(\alph*)}]
			\item If $\mathbf{T}_{\varphi}$ is hypercyclic on the Hardy space, then
			\begin{enumerate}[label = \textup{(\arabic*)}]
				\item the function $\widetilde{\varphi}$ is univalent in $\DD_0$;
				\item $\overline{\DD} \cap [\CC \backslash \widetilde{\varphi} (\DD_0)] \neq \varnothing$ and $\widehat{\DD} \cap [\CC \backslash \widetilde{\varphi} (\DD_0)] \neq \varnothing$.
			\end{enumerate}
			\item Assume in addition that $p$ is continuous on the closed disk  $\overline{\DD}$ and that
			\begin{enumerate}[label = \textup{(\arabic*$'$)}]
				\item the function $\widetilde{\varphi}$ is univalent in $\overline{\DD}_0$;
				\item $\DD \cap [\CC \backslash \widetilde{\varphi} (\DD_0)] \neq \varnothing$ and $\widehat{\DD} \cap [\CC \backslash \widetilde{\varphi} (\DD_0)] \neq \varnothing$.
			\end{enumerate}
			Then $\mathbf{T}_{\varphi}$ is hypercyclic on the Hardy space.
		\end{enumerate}
	\end{thm}

	For subsequent developments on hypercyclic Toeplitz operators on the Hardy space, the reader is referred to \cite{ABCL, GLS, DE, FGO}. Following the establishment of the theory of hypercyclic Toeplitz operators on the Hardy space, we naturally turn our attention to another important space of analytic functions, namely the Bergman space. Let $\Hol (\DD)$ denote the space of analytic functions on $\DD$. The \textit{Bergman space} is defined by
	\begin{align*}
		L_a^ 2 = L^ 2 (\DD, \rd A) \cap \Hol (\DD),
	\end{align*}
	where $\rd A$ is the normalized Lebesgue area measure on $\DD$.
	Clearly, $L_a^ 2$ is a closed subspace of $L^ 2 (\DD, \rd A)$.  Letting $P$ be the orthogonal projection from $L^ 2 (\DD, \rd A)$ onto $L_a^ 2$, it can be shown that $P$ is an integral operator given by
	\begin{align}
		P f (z) = \int_{\DD} \frac{f(w)}{(1 - z \overline{w})^ 2} \dd A (w), \qquad z \in \DD. \label{eq10}
	\end{align}
	For $\varphi \in L^ {\infty} (\DD, \rd A)$, the \textit{Toeplitz operator with symbol $\varphi$} on the Bergman space, which is also referred to as the Bergman--Toeplitz operator with symbol $\varphi$, is defined by
	\begin{align}
		T_{\varphi} f (z) = P (\varphi f) (z) = \int_{\DD} \frac{\varphi (w) f (w)}{(1 - z \overline{w})^ 2} \dd A (w), \qquad z \in \DD. \label{eq3}
	\end{align}
	Like Hardy--Toeplitz operators, Bergman--Toeplitz operators are also bounded linear operators and satisfy all of the identities in (\ref{eq4}) and the inequality in  (\ref{ineq3}) with $\| \varphi \|_{L^ {\infty} (\TT, \rd m)}$ replaced by $\| \varphi \|_{L^ {\infty} (\DD, \rd A)}$. One can consult \cite{Zhu} for detailed discussions of Bergman spaces and associated Toeplitz operators.
	
	In the remainder of this paper, for a function $\varphi$ continuous on $\overline{\DD}$, we use $\mathbf{T}_{\varphi} := \mathbf{T}_{\varphi |_{\TT}}$ and $T_{\varphi}$ to denote Toeplitz operators on the Hardy and Bergman spaces, respectively. By ``Toeplitz operators with the same symbol'', we mean the operators $\mathbf{T}_{\varphi}$ and $T_{\varphi}$ for some $\varphi$ continuous on $\overline{\DD}$.
	
	Recently,  hypercyclicity of Bergman--Toeplitz operators with some harmonic polynomial symbols was considered by the first and third authors in \cite{LZ}, and several necessary or sufficient conditions were obtained. In particular, it was shown that $T_{a \overline{z} + c_0 + c_1 z}$ is hypercyclic if $|a| > (3 + \sqrt{2}) |c_1|$ and the image of $\DD$ under the symbol $a \overline{z} + c_0 + c_1 z$ intersects $\TT$. The coefficient $3 + \sqrt{2}$ seems to be the best possible in the approach of that paper. Notably, this result was improved by Abad \cite{Aba}, in which the constant $3 + \sqrt{2}$ is replaced by $1$. As a consequence, Abad obtained the following necessary and sufficient condition:
	
	\begin{thm}[{\cite[Theorem 3.6]{Aba}}] \label{theorem: Abad's result}
		Suppose that $\varphi (z) = a \overline{z} + c_0 + c_1 z$, where $a, c_0, c_1 \in \CC$ and $a \neq 0$. Then $T_{\varphi}$ is hypercyclic on the Bergman space if and only if $|a| > |c_1|$ and $\varphi (\DD) \cap \TT \neq \varnothing$.
	\end{thm}

	Abad's proof relies on the theory of orthogonal polynomials. However, this approach is no longer applicable when the analytic part of the symbol $\varphi(z) = a \overline{z} + p(z)$ has degree greater than or equal to 2. In this paper, we continue the study of hypercyclic Toeplitz operators with symbols $\varphi (z) = a \overline{z} + p (z)$ on the Bergman space, where $a \in \CC\backslash \{ 0 \}$ and $p$ is an  analytic function  on $\overline{\DD}$. After recalling the necessary preliminaries, in Section \ref{Dynamical properties} we first establish a necessary condition for the Bergman--Toeplitz operator \(T_{\varphi} = T_{a \overline{z} + p}\), where \(a \neq 0\) and \(p\) is analytic on \(\overline{\mathbb{D}}\), to be hypercyclic, and then present a sufficient condition for its hypercyclicity, as well as other dynamical properties.
It is worth noting that the methods in \cite{LZ, Aba} rely on recurrence relations for the Taylor coefficients of eigenvectors of $T_{a \overline{z} + p}$. In this paper, by contrast, we mainly use spectral theory to study the hypercyclicity of the Toeplitz  operator $T_{a \overline{z} + p}$. This illustrates how the hypercyclicity of Bergman-Toeplitz operators with harmonic symbols is closely related to their spectral properties in a general setting. As a particular application, we provide a new proof of Theorem \ref{theorem: Abad's result}.

	Comparing Theorems \ref{theorem: Shkarin's result} and \ref{theorem: Abad's result}, we observe that when $\varphi$ is a harmonic linear polynomial, $\mathbf{T}_{\varphi}$ and $T_{\varphi}$ share the same hypercyclicity. This is partly due to the fact that $\mathbf{T}_{\varphi}$ and $T_{\varphi}$ have identical spectra (see \cite[Theorem 3.1]{GZZ}). In fact, without requiring the symbol to be harmonic, it is not hard to construct a function $\Psi$ such that the Hardy--Toeplitz operator $\mathbf{T}_{\Psi}$ is non-hypercyclic but the corresponding Bergman--Toeplitz operator $T_{\Psi}$ is hypercyclic, see Proposition \ref{NHHB} for the details. On the other hand, in the final section of this paper, we can construct a more interesting example: there exists a bounded harmonic function $\Phi$ such that $\mathbf{T}_{\Phi}$ is hypercyclic on the Hardy space while $T_{\Phi}$ is non-hypercyclic on the Bergman space.  The underlying reason is that, for a general bounded harmonic symbol, the spectra of Toeplitz operators with the same symbol do not always coincide.
The technical difficulty of this construction is as follows. For bounded harmonic symbols of the form $a\overline{z}+p(z)$, we know that there exists such a symbol whose associated Bergman--Toeplitz operator has a disconnected spectrum (see \cite[Theorem 4.1]{GZZ}); by Theorem \ref{theorem: necessary condition for hypercyclicity} of the present paper, this operator is non-hypercyclic. However, determining whether the Hardy--Toeplitz operator with the same symbol is hypercyclic is far from straightforward.
The crux of the construction is to observe that, according to Theorem \ref{theorem: hypercyclic Hardy-Toeplitz operators}, hypercyclicity of Hardy--Toeplitz operators reduces essentially to the univalence of $\widetilde{\varphi} (z)  = a z^ {-1} + p (z)$ on $\DD_0$. This leads us to employ tools from univalent function theory. More precisely, the symbol $\Phi$ introduced later arises from a special class of univalent analytic functions. See Theorem \ref{theorem: an example} for details.
	
	\section{Preliminaries}
	
	\noindent \textbf{2.1.~Linear dynamics}
	\vspace{5pt}
	
	Let $X$ be a separable Banach space and $\mathfrak{L} (X)$ the space of all bounded linear operators on $X$. The identity operator on a Banach space will be always denoted by $I$. An operator $T \in \mathfrak{L} (X)$ is said to be \textit{hypercyclic} if there exists a vector $x \in X$ such that the $T$-orbit of $x$, namely
	\begin{align*}
		\orb{(x, T)} = \{ T^ n x: n \in \NN \},
	\end{align*}
	is dense in $X$. Hypercyclic operators enjoy the following spectral property.
	
	\begin{thm}[{\cite[Theorem 1.18]{BM}}] \label{theorem: spectra of hypercyclic operators}
		Let $X$ be a separable Banach space, and let $T \in \mathfrak{L} (X)$ be hypercyclic. Then every connected component of the spectrum of $T$ intersects $\TT$.
	\end{thm}
	
	There are also some important strengthenings of hypercyclicity. Let $T \in \mathfrak{L} (X)$.
	\begin{itemize}
		\item $T$ is said to be \textit{frequently hypercyclic} if there exists an $x \in X$ satisfying the following: for every nonempty open subset $V$ of $X$, there is a strictly increasing sequence $\{ n_k \}_{k = 1}^ {\infty}$ of $\NN_+$ and a positive integer $C$ such that
		\begin{align*}
			T^ {n_k} x \in V \qquad \textrm{and} \qquad n_k \leqslant C k
		\end{align*}
		for all $k \in \NN_+$.
		\item $T$ is said to be (\textit{topologically}) \textit{mixing} if for every pair $(U, V)$ of nonempty open subsets of $X$, there exists a positive integer $N$ such that
		\begin{align*}
			T^ n (U) \cap V \neq \varnothing, \qquad n \geqslant N.
		\end{align*}
		\item $T$ is said to be \textit{weakly mixing} if for every pair $(U, V)$ of nonempty open subsets of $X$ and for every positive integer $N$, there exists some $s \in \NN_+$ such that
		\begin{align*}
			T^ n (U) \cap V \neq \varnothing, \qquad s \leqslant n \leqslant s + N.
		\end{align*}
		\item To define chaotic operators, we first introduce the concept of periodic points. A vector $x \in X$ is a \textit{periodic point} of $T$ if $T^ n x = x$ for some $n \in \NN_+$. The set of all periodic points of $T$ will be denoted by $\Per(T)$. Then we say that $T$ is \textit{chaotic} if it is hypercyclic and if $\Per(T)$ is dense in $X$.
	\end{itemize}
	
	For detailed discussions of hypercyclicity and other dynamical properties mentioned above, the reader is referred to \cite{BM}. By the definitions above, both chaoticity and frequent hypercyclicity imply hypercyclicity. Using Birkhoff's transitivity theorem (see \cite[Theorem 1.2]{BM}), it is also easy to see that both mixing property and weak mixing property are stronger than hypercyclicity. In addition, it can be shown that chaoticity is stronger than weak mixing property (see \cite[Corollary 6.11]{BM}), and it is clear from the definitions that mixing property is stronger than weak mixing property. We summarize these conclusions in a lemma for ease of reference.
	
	\begin{lem} \label{lemma: dynamical properties}
		Suppose that $X$ is a separable Banach space and $T \in \mathfrak{L} (X)$. The five dynamical properties of $T$ defined above satisfy:
		\begin{align*}
			\begin{matrix}
				\textrm{chaoticity} & \Longrightarrow & \textrm{weak mixing property} & \Longrightarrow & \hspace{-2em} \textrm{hypercyclicity} \\[4pt]
				& & \big\Uparrow & & \hspace{-2em} \big\Uparrow \\[4pt]
				& & \textrm{mixing property} & & \hspace{-2em} \textrm{frequent hypercyclicity.}
			\end{matrix}
		\end{align*}
	\end{lem}
	
	In the remainder of this subsection, we give some useful characterizations of chaoticity, mixing property and frequent hypercyclicity, each of which requires the operator to have a large supply of eigenvectors. The Godefroy--Shapiro criterion is a classical result for testing the hypercyclicity of an operator. In fact, a careful inspection of its proof reveals that it further provides a criterion for mixing operators.
	
	\begin{thm}[Shapiro--Godefroy criterion] \label{lemma: S-G criterion}
		Let $X$ be a separable Banach space and $T \in \mathfrak{L} (X)$. If both
		\begin{align}
			\spa{\Bigg[ \bigcup_{|\lambda| < 1} {\ker{(T - \lambda I)}} \Bigg]} \qquad \textrm{and} \qquad \spa{\Bigg[ \bigcup_{|\lambda| > 1} {\ker{(T - \lambda I)}} \Bigg]} \label{for1}
		\end{align}
		are dense in $X$, then $T$ is mixing.
	\end{thm}
	
	\begin{proof}
		According to the proof of \cite[Corollary 1.10]{BM}, $T$ satisfies the hypercyclicity criterion  (see \cite[Definition 1.5]{BM}) with respect to the sequence of positive integers $\{ 1, 2, \dots \}$. Hence, by the argument in \cite[Section 2.1]{BM}, $T$ is mixing.
	\end{proof}
	
	A further consideration of eigenvectors associated with eigenvalues of modulus $1$ gives a characterization of chaoticity.
	
	\begin{lem} \label{lemma: chaotic operators}
		Let $X$ be a separable Banach space and $T \in \mathfrak{L} (X)$. If subspaces given by \textup{(\ref{for1})} and the subspace
		\begin{align*}
			\spa{\Bigg[ \bigcup_{t \in \QQ} {\ker{(T - \ee^ {2 \ppi\ii t} I)}} \Bigg]}
		\end{align*}
		are all dense in $X$, then $T$ is chaotic.
	\end{lem}
	
	\begin{proof}
		According to \cite[Remark 6.7]{BM}, we have
		\begin{align*}
			\Per(T) = \spa{\Bigg[ \bigcup_{t \in \QQ} {\ker{(T - \ee^ {2 \ppi \ii t} I)}} \Bigg]}.
		\end{align*}
		Thus, $\Per(T)$ is dense in $X$. On the other hand, by Theorem \ref{lemma: S-G criterion}, $T$ is also mixing, and hence hypercyclic. Therefore, $T$ is chaotic.
	\end{proof}
	
	Suppose that $\rd \sigma$ is a Borel probability measure on $\TT$. For an operator $T \in \mathfrak{L} (X)$, we say that $T$ has a \textit{$\sigma$-spanning set of $\TT$-eigenvectors} if for every Borel subset $A$ of $\TT$ with $\sigma (A) = 1$, the subspace
	\begin{align*}
		\spa{\Bigg[ \bigcup_{\lambda \in A} {\ker{(T - \lambda I)}} \Bigg]}
	\end{align*}
	is dense in $X$. If $T$ has a $\sigma$-spanning set of $\TT$-eigenvectors for some continuous measure $\rd \sigma$ (i.e., $\sigma (\{ \lambda \}) = 0$ for every $\lambda \in \TT$), then we say that $T$ has a \textit{perfectly \textup{(}$\sigma$-\textup{)}spanning set of $\TT$-eigenvectors}. With this terminology established, we give the following sufficient condition for frequent hypercyclicity.
	
	\begin{lem}[{\cite[Corollary 6.24]{BM}}] \label{lemma: frequently hypercyclic operators}
		Let $H$ be a separable Hilbert space and $T \in \mathfrak{L} (H)$. If $T$ has a perfectly spanning set of $\TT$-eigenvectors, then $T$ is frequently hypercyclic.
	\end{lem}

	\noindent \textbf{2.2.~Toeplitz operators}
	\vspace{5pt}
	
	Throughout this paper, we fix the following orthonormal basis $\{ e_j \}_{j = 0}^{\infty}$ of the Bergman space:
	\begin{align*}
		e_j (z) = \sqrt{j + 1} z^ j, \qquad j \in \NN.
	\end{align*}
	The winding number of a curve $\varphi (\ee^ {\ii \theta})$ ($0 \leqslant \theta \leqslant 2 \ppi$) with respect to $\lambda \in \CC \backslash \varphi (\TT)$ will be denoted by $\wind{(\varphi (\TT), \lambda)}$. For a harmonic function $h (z) = q (\overline{z}) + p (z)$ continuous on $\overline{\DD}$, where $q$ is an analytic polynomial and $p$ is an analytic function continuous on $\overline{\DD}$, we denote the analytic extension of $h |_{\TT}$ by $\widetilde{h}$:
	\begin{align*}
		\widetilde{h} (z) = q \Big( \frac{1}{z} \Big) + p (z), \qquad z \in \overline{\DD}_0.
	\end{align*}
	Thus, $\widetilde{h} |_{\TT} = h |_{\TT}$, and this relation will be frequently used in the remainder of this paper.

	\vspace{5pt}
	\noindent \textit{2.2.1. Bergman--Toeplitz operators with monomial symbols}
	\vspace{5pt}
	
	Since we mainly deal with Toeplitz operators with harmonic symbols, the following elementary lemma which calculates Toeplitz operators with monomial symbols is useful.
	
	\begin{lem} \label{lemma: Toeplitz operators with monomial symbols}
		For $k \in \NN$ and $j \in \NN$, we have
		\begin{align*}
			T_{\overline{z}^ k} e_j =
			\begin{cases}
				0, & j < k, \vspace{2mm}\\
				\sqrt{\frac{j - k + 1}{j + 1}} e_{j - k}, & j \geqslant k,
			\end{cases}
		\end{align*}
		and
		\begin{align*}
			T_{z^ k} e_j = \sqrt{\frac{j + 1}{j + k + 1}} e_{j + k}, \qquad j \geqslant 0.
		\end{align*}
	\end{lem}
	
	\begin{proof}
		Using (\ref{eq3}), for $k \in \NN$ and $j \in \NN$ we have
		\begin{align*}
			T_{z^ k} e_j (z) & = \int_{\DD} \frac{e_j (w) w^ k}{(1 - z \overline{w})^ 2} \dd A(w) = \sqrt{j + 1} \sum_{i = 0}^ {\infty} {(i + 1) z^ i \int_{\DD} w^ {j + k} \overline{w}^ i \dd A(w)} \\
			& = \sqrt{j + 1} (j + k + 1) z^ {j + k} \int_{\DD} |w|^ {2 (j + k)} \dd A (w) = \sqrt{\frac{j + 1}{j + k + 1}} e_{j + k} (z).
		\end{align*}
		The calculation of $T_{\overline{z}^ k} e_j$ is similar, so we omit the details.
	\end{proof}

	\noindent \textit{2.2.2. Spectrum and its connectedness}
	\vspace{5pt}
	
	For a bounded operator $A$ on a Hilbert space, we use $\sigma (A)$, $\sigma_{\mathrm{p}} (A)$ and $\sigma_{\mathrm{e}} (A)$ to denote its spectrum, point spectrum and essential spectrum, respectively. Let $\Hol (\overline{\DD})$ be the space of functions analytic on $\overline{\DD}$, i.e., $f \in \Hol (\overline{\DD})$ if and only if there exists some open set $U$ containing $\overline{\DD}$ such that $f$ is analytic on $U$. The following characterization of the spectra of Toeplitz operators plays a foundational role in our analysis.
	
	\begin{thm}[{\cite{GZZ}}] \label{theorem: spectrum and essential spectrum}
		Let $p \in \Hol (\overline{\DD}) $, $a \in \CC \backslash \{ 0 \}$, and $\varphi (z) = a \overline{z} + p(z)$. Then:
		\begin{enumerate}[label = \textup{(\alph*)}]
			\item $\sigma_{\mathrm{e}} (T_{\varphi}) = \varphi (\TT)$.
			\item For every $\lambda \notin \varphi (\TT)$,
			\begin{align*}
				\ind{(T_{\varphi} - \lambda I)} = -\wind{(\varphi (\TT), \lambda)}.
			\end{align*}
			\item For every $\lambda \notin \varphi (\TT)$, $\lambda$ is an eigenvalue of $T_{\varphi}$ if and only if either $\widetilde{\varphi} (z) - \lambda$ is non-vanishing on $\DD_0$ or $\widetilde{\varphi} (z) - \lambda$ has finitely many simple zeros $\{ z_1, \dots, z_k \}$ in $\DD_0$ which satisfy
			\begin{align*}
				\frac{z_j^ 2 p' (z_j)}{a} = \frac{n_j + 2}{n_j + 1}
			\end{align*}
			for some $n_j \in \NN$, where $j = 1, 2, \dots, k$.
		\end{enumerate}
	\end{thm}

	Let $a \in \CC \backslash \{ 0 \}$ and $p \in \Hol (\overline{\DD})$. According to Theorem \ref{theorem: spectrum and essential spectrum}, the spectrum of $T_{\varphi}$ with $\varphi (z) = a \overline{z} + p(z)$  can be decomposed into the following three disjoint parts:
	\begin{align*}
		\sigma (T_{\varphi} )& = \varphi (\TT) \sqcup \{\lambda\in \CC \backslash \varphi (\TT): \wind{(\varphi (\TT), \lambda)} \neq 0 \} \\
		& \quad \, \sqcup \{\lambda\in \CC \backslash \varphi (\TT): \wind{(\varphi (\TT), \lambda)} = 0 \ \textrm{ and } \ \ker{(T_{\varphi} - \lambda)} \neq \{ 0  \} \}.
	\end{align*}
	Let
	\begin{align*}
		\Lambda_{\mathrm{c}} (T_{\varphi}) = \varphi (\TT) \sqcup \{ \lambda \in \CC \backslash \varphi (\TT) : \wind{(\varphi (\TT), \lambda)} \neq 0 \}
	\end{align*}
	and
	\begin{align*}
		\Lambda_{\mathrm{d}} (T_{\varphi}) = \{\lambda\in \CC \backslash \varphi (\TT): \wind{(\varphi (\TT), \lambda)} = 0 \ \textrm{ and } \ \ker{(T_{\varphi} - \lambda I)} \neq \{ 0  \} \}.
	\end{align*}
	$\Lambda_{\mathrm{c}} (T_{\varphi})$ is clearly a connected subset of $\sigma (T_{\varphi})$. But for $\Lambda_{\mathrm{d}} (T_{\varphi})$, we have:
	
	\begin{lem} \label{lemma: Lambdad is discrete}
		Every point in $\Lambda_{\mathrm{d}} (T_{\varphi})$ is an isolated point of $\sigma (T_{\varphi})$. Consequently, $\Lambda_{\mathrm{c}} (T_{\varphi})$ is a connected component of $\sigma (T_{\varphi})$.
	\end{lem}
	
	\begin{proof}
		It suffices to consider the case that $\Lambda_{\mathrm{d}} (T_{\varphi})$ is nonempty. Pick a $\lambda_0 \in \Lambda_{\mathrm{d}} (T_{\varphi})$. By the argument principle, the equation
		\begin{align*}
			\widetilde{\varphi} (z) = \frac{a}{z} + p (z) = \lambda_0
		\end{align*}
		has a unique simple root $z_0$ in $\DD_0$. Since $\lambda_0$ is an eigenvalue of $T_{\varphi}$, according to (c) of Theorem \ref{theorem: spectrum and essential spectrum}, we must have
		\begin{align*}
			\frac{z_0^ 2 p' (z_0)}{a} = \frac{m_0 + 2}{m_0 + 1}
		\end{align*}
		for some $m_0 \in \NN$. As a consequence,
		\begin{align*}
			(\widetilde{\varphi})' (z_0) = -\frac{a}{z_0^ 2} + p' (z_0) = -\frac{a}{z_0^ 2} + \frac{m_0 + 2}{m_0 + 1} \frac{a}{z_0^ 2} = \frac{1}{m_0 + 1} \frac{a}{z_0^ 2} \neq 0.
		\end{align*}
		By the inverse function theorem for analytic functions, $\widetilde{\varphi}$ has an analytic inverse $(\widetilde{\varphi})^ {-1}$ on a neighborhood of $z_0$. Now, let $r > 0$ be a small number such that
		\begin{enumerate}[label = (\roman*)]
			\item $D (z_0, r) \subseteq \DD_0$,
			\item $\widetilde{\varphi}$ has an analytic inverse on $D (z_0, r)$, and
			\item $U := \widetilde{\varphi} (D (z_0, r))$ lies in the connected component of $\CC \backslash \varphi (\TT)$ that contains $\lambda_0$ (note that $\lambda_0 \in \CC \backslash \varphi (\TT)$).
		\end{enumerate}
		By Condition (iii), for each $\lambda \in U$ we have
		\begin{align}
			\wind{(\varphi (\TT), \lambda)} = \wind{(\varphi (\TT), \lambda_0)} = 0, \label{eq5}
		\end{align}
		and hence by the argument principle, the equation
		\begin{align*}
			\widetilde{\varphi} (z) = \frac{a}{z} + p(z) = \lambda
		\end{align*}
		has $(\widetilde{\varphi})^ {-1} (\lambda)$ as its unique simple root in $\DD_0$.
		According to (\ref{eq5}) and Theorem \ref{theorem: spectrum and essential spectrum}, $\lambda \in U$ is in $\sigma (T_{\varphi})$ if and only if
		\begin{align}
			G (\lambda) := \frac{\big[ (\widetilde{\varphi})^ {-1} (\lambda) \big]^ 2 p' \big( (\widetilde{\varphi})^ {-1} (\lambda) \big)}{a} = \frac{m + 2}{m + 1} \label{eq6}
		\end{align}
		for some $m \in \NN$. Note that
		\begin{align*}
			G (\lambda_0) = \frac{z_0^ 2 p' (z_0)}{a} = \frac{m_0 + 2}{m_0 + 1}
		\end{align*}
		as $\widetilde{\varphi} (z_0) = \lambda_0$. Since $G$ is non-constant, we can find a sufficiently small neighborhood $V \subseteq U$ of $\lambda_0$ such that
		\begin{align*}
			\left\{ G (\lambda) : \lambda \in V \backslash \{ \lambda_0 \} \right\} \cap \left\{ \frac{j + 2}{j + 1} : j \in \NN \right\} = \varnothing.
		\end{align*}
		Therefore, every $\lambda \in V \backslash \{ \lambda_0 \}$ is not in $\sigma (T_{\varphi})$, since (\ref{eq6}) does not hold for this $\lambda$. This implies that $\lambda_0$ is an isolated point of $\sigma (T_{\varphi})$.
	\end{proof}
	
	As an immediate consequence of the preceding lemma, we obtain that:
	
	\begin{lem} \label{lemma: connectedness}
		Let $p \in \Hol (\overline{\DD})$, $a \in \CC \backslash \{ 0 \}$, and $\varphi (z) = a \overline{z} + p(z)$. Then $\sigma (T_{\varphi})$ is connected if and only if $\Lambda_{\mathrm{d}} (T_{\varphi}) = \varnothing$.
	\end{lem}

	\noindent \textit{2.2.3. The univalent case}
	\vspace{5pt}

	As before, we  let $\varphi (z) = a \overline{z} + p (z)$, where $a \in \CC \backslash \{ 0 \}$ and $p \in \Hol (\overline{\DD})$, and assume additionally that $\widetilde{\varphi} (z) = a z^ {-1} + p (z)$ is univalent on $\DD_0$. Let $\Omega_{\infty}$ be the unbounded connected component of $\CC \backslash \varphi (\TT)$ and let $\{ \Omega_j \}_{j \geqslant 0}$ be the collection of bounded connected components of $\CC \backslash \varphi (\TT)$. We point out that it is possible that $\{ \Omega_j \}_{j \geqslant 0}$ is empty, i.e., $\CC \backslash \varphi (\TT)$ has no bounded connected components. The simplest example is the function
	\begin{align*}
		\varphi (z) = \overline{z} + z.
	\end{align*}
	It is easy to check that $\widetilde{\varphi} (z) = z^ {-1} + z$ is univalent on $\DD_0$. However, $\varphi (\TT) = [-2, 2]$, showing that $\CC \backslash \varphi (\TT) = \CC \backslash [-2, 2]$ is connected and unbounded.
	
	\begin{lem} \label{lemma: univalent}
		Let $a \in \CC \backslash \{ 0 \}$, $p \in \Hol (\overline{\DD})$, and $\varphi (z) = a \overline{z} + p(z)$. Suppose that $\widetilde{\varphi}$ is univalent on $\DD_0$. Then:
		\begin{enumerate}[label = \textup{(\alph*)}]
			\item $\widetilde{\varphi}$ maps $\DD_0$ onto $\Omega_{\infty}$ conformally.\\
			\item $\wind{(\varphi (\TT), \lambda)} = -1$ for every $\lambda \in \bigsqcup\limits_{j \geqslant 0} {\Omega_j}$.\\
			\item $\Lambda_{\mathrm{c}} (T_{\varphi}) = \CC \backslash \Omega_{\infty} = \varphi (\TT) \sqcup \big( \bigsqcup\limits_{j \geqslant 0} {\Omega_j} \big)$.
		\end{enumerate}
	\end{lem}
	
	\begin{proof}
		We first prove (a). We claim that
		\begin{align}
			\widetilde{\varphi} (\DD_0) \cap \widetilde{\varphi} (\TT) = \varnothing. \label{eq11}
		\end{align}
		Otherwise, there exist $z_0 \in \DD_0$ and $z_1 \in \TT$ such that
		\begin{align*}
			\widetilde{\varphi} (z_0) = \widetilde{\varphi} (z_1) =: \xi.
		\end{align*}
		Let $r_0$ be a small positive number such that $D (z_0, r_0) \subseteq \DD_0$. By the open mapping theorem of analytic functions, $\widetilde{\varphi} (D (z_0, r_0))$ is an open neighborhood of $\xi$. Then by the continuity of $\widetilde{\varphi}$, there exists an $r_1 > 0$ such that
		\begin{align}
			\widetilde{\varphi} (D (z_1, r_1)) \subseteq \widetilde{\varphi} (D (z_0, r_0)). \label{ineq2}
		\end{align}
		We may assume that $r_1$ is sufficiently small such that
		\begin{align}
			D (z_0, r_0) \cap D (z_1, r_1) = \varnothing. \label{eq12}
		\end{align}
		Now using (\ref{ineq2}), we can find $w_0 \in D (z_0, r_0)$, $w_1 \in D (z_1, r_1) \cap \DD_0$, such that $\widetilde{\varphi} (w_0) = \widetilde{\varphi} (w_1)$. By (\ref{eq12}), $w_0 \neq w_1$, so we deduce that $\widetilde{\varphi}$ is not univalent on $\DD_0$, which contradicts the hypotheses.
		
		The set $\widetilde{\varphi} (\DD_0)$ is connected (since $\DD_0$ is connected and $\widetilde{\varphi}$ is continuous), unbounded and contained in $\CC \backslash \varphi (\TT)$ (due to (\ref{eq11})). Hence it is contained in $\Omega_{\infty}$. On the other hand, we also have by the basic topological fact that
		\begin{align*}
			\wind{(\varphi (\TT), \lambda)} = 0 \qquad (\lambda \in \Omega_{\infty}),
		\end{align*}
		and hence by the argument principle that $\Omega_{\infty} \subseteq \widetilde{\varphi} (\DD_0)$. We conclude that
		\begin{align*}
			\widetilde{\varphi} (\DD_0) = \Omega_{\infty}.
		\end{align*}
		Since $\widetilde{\varphi}$ is univalent on $\DD_0$, $\widetilde{\varphi}$  maps $\DD_0$  conformally onto $\Omega_{\infty}$.
		
		Next, we are going to prove (b). We only consider the non-vacuous case that $\{ \Omega_j \}_{j \geqslant 0}$ is nonempty. Let $j \geqslant 0$ and pick an arbitrary $\lambda \in \Omega_j$. Since
		\begin{align*}
			\widetilde{\varphi} (\DD_0) \cap \Omega_j = \Omega_{\infty} \cap \Omega_j = \varnothing,
		\end{align*}
		the function $\widetilde{\varphi} (z) - \lambda$ is non-vanishing on $\DD_0$. It follows from the argument principle that
		\begin{gather*}
			\wind{(\varphi (\TT), \lambda)} = -1.
		\end{gather*}
	
		Lastly, observe that (c) is a direct consequence of the definition of $\Lambda_{\mathrm{c}}$ and (b). The proof is complete.
	\end{proof}

	We next give a characterization for the connectedness of $\sigma (T_{\varphi})$. Let
	\begin{align*}
		D_{\varphi} = \left\{ \frac{z^ 2 p' (z)}{a} : z \in \DD_0 \right\}
	\end{align*}
	and
	\begin{align}\label{W}
		W = \left\{ \frac{m + 2}{m + 1} : m \in \NN \right\}.
	\end{align}

	\begin{lem} \label{lemma: connectedness for univalent case}
		Let $a \in \CC \backslash \{ 0 \}$, $p \in \Hol (\overline{\DD})$, and $\varphi (z) = a \overline{z} + p(z)$. Suppose that $\widetilde{\varphi}$ is univalent on $\DD_0$. Then $\sigma (T_{\varphi})$ is connected if and only if $D_{\varphi} \cap W = \varnothing$.
	\end{lem}

	\begin{proof}
		In view of Lemma \ref{lemma: connectedness}, we need to show that
		\begin{align*}
			\Lambda_{\mathrm{d}} (T_{\varphi}) & = \left\{ \lambda \in \CC \backslash \varphi (\TT) : \wind{(\varphi (\TT), \lambda)} = 0  \ \textrm{ and } \  \ker{(T_{\varphi} - \lambda I)} \neq \{ 0 \} \right\} \\
			& = \left\{ \lambda \in \Omega_{\infty} : \textrm{$\lambda$ is an eigenvalue of $T_{\varphi}$} \right\}
		\end{align*}
		is empty if and only if $D_{\varphi} \cap W = \varnothing$,
		where the second ``$=$'' is due to (b) of Lemma \ref{lemma: univalent}.
		According to (a) of Lemma \ref{lemma: univalent}, for every $\lambda \in \Omega_{\infty}$, the equation
		\begin{align*}
			\widetilde{\varphi} (z) = \lambda
		\end{align*}
		has a unique root in $\DD_0$, namely $(\widetilde{\varphi})^ {-1} (\lambda)$. Thus, by (c) of Theorem \ref{theorem: spectrum and essential spectrum}, $\lambda \in \Omega_{\infty}$ is an eigenvalue of $T_{\varphi}$ if and only if
		\begin{align*}
			\frac{\big[ (\widetilde{\varphi})^ {-1} (\lambda) \big]^ 2 p' \big( (\widetilde{\varphi})^ {-1} (\lambda) \big)}{a} = \frac{m + 2}{m + 1}
		\end{align*}
		holds for some $m \in \NN$. As a result, $\Omega_{\infty}$ contains no eigenvalues of $T_{\varphi}$ if and only if
		\begin{align*}
			\left\{ \frac{\big[ (\widetilde{\varphi})^ {-1} (\lambda) \big]^ 2 p' \big( (\widetilde{\varphi})^ {-1} (\lambda) \big)}{a} : \lambda \in \Omega_{\infty} \right\} \cap W = \varnothing.
		\end{align*}
		Observe that as $\lambda$ ranges over $\Omega_{\infty}$, $(\widetilde{\varphi})^ {-1} (\lambda)$ runs over $\DD_0$, to obtain
		\begin{align*}
			D_{\varphi} = \left\{ \frac{\big[ (\widetilde{\varphi})^ {-1} (\lambda) \big]^ 2 p' \big( (\widetilde{\varphi})^ {-1} (\lambda) \big)}{a} : \lambda \in \Omega_{\infty} \right\}.
		\end{align*}
		Therefore, we arrive at the conclusion that $\sigma (T_{\varphi})$ is connected if and only if $D_{\varphi} \cap W = \varnothing$.
	\end{proof}

	In view of the above lemma, for the symbol $\varphi$ given at the beginning of this subsection, we say that $T_{\varphi}$ is \textit{$S$-spectrum connected} (with $S$ standing for ``strong'') if there exists a $\delta > 0$ such that
	\begin{align*}
		\dist{(D_{\varphi}, W)} \geqslant \delta.
	\end{align*}
	
	\noindent \textit{2.2.4. A spectrum-preserving transform}
	\vspace{5pt}
	
	Let $V, J: L^ 2 (\DD, \rd A) \longrightarrow L^ 2 (\DD, \rd A)$ be the following two operators:
	\begin{align*}
		V f (z) = \overline{f (\overline{z})} \qquad \textrm{and} \qquad J f (z) = f (\overline{z}).
	\end{align*}
	Clearly, $V$ is antiunitary, $J$ is unitary, both $V$ and $J$ are idempotent, and $L_a^ 2$ is an invariant subspace of $V$. Some basic properties of $V$ are collected in the next lemma.
	
	\begin{lem} \label{lemma: properties of V}
		The following are true:
		\begin{enumerate}[label = \textup{(\alph*)}]
			\item  $V (f g) = (V f) (V g)$ for $f \in L^ {\infty} (\DD, \rd A)$ and $g \in L^ 2 (\DD, \rd A)$.
			\item $V P = P V$ on $L^ 2 (\DD, \rd A)$.
			\item $T_{V \varphi} = V T_{\varphi} V$ on $L_a^ 2$ for every $\varphi \in L^ {\infty} (\DD, \rd A)$.
		\end{enumerate}
	\end{lem}
	
	\begin{proof}
		(a) is obvious. For an arbitrary $f \in L^ 2 (\DD, \rd A)$, we have by (\ref{eq10}) that
		\begin{align*}
			V P f (z) & = \overline{(P f) (\overline{z})} = \overline{\int_{\DD} \frac{f (w)}{(1 - \overline{z w})^ 2} \dd A (w)} = \int_{\DD} \frac{\overline{f (w)}}{(1 - z w)^ 2} \dd A (w) \\
			& = \int_{\DD} \frac{\overline{f (\overline{w})}}{(1 - z \overline{w})^ 2} \dd A(w) = P V f (z), \qquad z \in \DD,
		\end{align*}
		to obtain (b). (c) is then an easy consequence of (a) and (b):
		\begin{gather*}
			T_{V \varphi} f = P [(V \varphi) f] = P [V (\varphi V f)] = V [P (\varphi V f)] = V T_{\varphi} V f, \qquad f \in L_a^ 2.
		\end{gather*}
		The proof of the lemma is complete.
	\end{proof}

	Using the preceding lemma, we easily derive the following useful property of $J$:
	
	\begin{lem} \label{lemma: property of J}
		Let $\varphi \in L^ {\infty} (\DD, \rd A)$. Then $\sigma (T_{\varphi}) = \sigma (T_{J \varphi})$.
	\end{lem}

	\begin{proof}
		Note that
		\begin{align*}
			(T_{\varphi} - \lambda I)^ * = T_{\overline{\varphi - \lambda}} = T_{V J (\varphi - \lambda)} = V T_{J (\varphi - \lambda)} V = V (T_{J \varphi} - \lambda I) V, \qquad \lambda \in \CC,
		\end{align*}
		in which the third ``$=$'' follows from (c) of Lemma \ref{lemma: properties of V}. Thus, we conclude that $T_{\varphi} - \lambda I$ is invertible if and only if $T_{J \varphi} - \lambda I$ is invertible. This finishes the proof.
	\end{proof}

	\section{Dynamical properties of  the Toeplitz operator \texorpdfstring{$T_{a \overline{z} + p}$}{T\_\{a {\textbackslash}overline\{z\} + p \}}} \label{Dynamical properties}
	
	Throughout this section, we let
	\begin{align*}
		\varphi (z) = a \overline{z} + p(z),
	\end{align*}
	where $a \in \CC \backslash \{ 0 \}$ and $p (z) = \sum\limits_{k = 0}^ {\infty} {c_k z^ k} \in \Hol (\overline{\DD})$. We begin with the following necessary condition for the operator $T_{\varphi}$ to be hypercyclic on the Bergman space $L_a^2$, which is analogous to (a) of Theorem \ref{theorem: hypercyclic Hardy-Toeplitz operators} in Section \ref{S1}.
	
	\begin{thm} \label{theorem: necessary condition for hypercyclicity}
		Let $a \in \CC \backslash \{ 0 \}$, $p \in \Hol (\overline{\DD})$, and $\varphi (z) = a \overline{z} + p(z)$. If $T_{\varphi}$ is hypercyclic, then
		\begin{enumerate}[label = \textup{(\alph*)}]
			\item $\widetilde{\varphi} (z) = a z^ {-1} + p (z)$ is univalent on $\DD_0$,
			\item $\overline{\DD} \cap \sigma (T_{\varphi}) \neq \varnothing$ and $\widehat{\DD} \cap \sigma (T_{\varphi}) \neq \varnothing$, and
			\item $\sigma (T_{\varphi})$ is connected.
		\end{enumerate}
	\end{thm}
	
	\begin{proof}
		Indeed, (a) has been proved in \cite[Corollary 3.2]{LZ} (in that corollary, $p$ is assumed to be an analytic polynomial, but the proof is identical if $p$ is replaced by a function analytic on $\overline{\DD}$).

To prove (c), by Lemma \ref{lemma: connectedness}, it suffices to show that $\Lambda_{\mathrm{d}} (T_{\varphi}) = \varnothing$. In fact, if $\Lambda_{\mathrm{d}} (T_{\varphi})$ is nonempty, we have
		\begin{align*}
			0 = -\wind{(\varphi (\TT), \lambda)} = \ind{(T_{\varphi} - \lambda I)} = \dim{\ker{(T_{\varphi} - \lambda I)}} - \dim{\ker{(T_{\varphi}^ * - \overline{\lambda} I)}}
		\end{align*}
for every $\lambda \in \Lambda_{\mathrm{d}} (T_{\varphi})$.
		Since $\lambda$ is an eigenvalue of $T_{\varphi}$, we derive that $\overline{\lambda}$ is an eigenvalue of $T_{\varphi}^ *$. However, this contradicts the hypercyclicity of $T_{\varphi}$ (see \cite[Proposition 1.17]{BM}).
		
		Lastly, let us prove (b). By (c) and Theorem \ref{theorem: spectra of hypercyclic operators}, we have
		\begin{align*}
			\varnothing \neq \TT \cap \sigma (T_{\varphi}) \subseteq \overline{\DD} \cap \sigma (T_{\varphi}).
		\end{align*}
		If $\widehat{\DD} \cap \sigma (T_{\varphi}) = \varnothing$, then we obtain
		\begin{align*}
			\varphi (\TT) & = \sigma_{\mathrm{e}} (T_{\varphi}) \subseteq \sigma (T_{\varphi}) \subseteq \overline{\DD},
		\end{align*}
		where the  ``$=$'' is due to (a) of Theorem \ref{theorem: spectrum and essential spectrum}.
		Hence,
		\begin{align*}
			\max_{z \in \TT} {|\varphi (z)|} \leqslant 1.
		\end{align*}
		It follows by the maximum modulus principle for harmonic functions that
		\begin{align*}
			\| T_{\varphi} \| \leqslant \sup_{z \in \DD} {|\varphi (z)|} \leqslant 1,
		\end{align*}
		which is impossible since the orbit of every vector under a contraction is bounded.
	\end{proof}

	We now turn to the sufficient condition for hypercyclicity and other dynamical properties of $T_{\varphi}$. Our main tools are Theorem \ref{lemma: S-G criterion} and Lemmas \ref{lemma: chaotic operators} and \ref{lemma: frequently hypercyclic operators}, so it is necessary to calculate eigenvectors of $T_{\varphi}$. Recall that
	\begin{align*}
		e_j (z) = \sqrt{j + 1} z^ j, \qquad j \in \NN.
	\end{align*}
	
	\begin{lem} \label{lemma: recurrence relation}
		Let $a \in \CC \backslash \{ 0 \}$, $p (z) = \sum\limits_{k = 0}^ {\infty} {c_k z^ k} \in \Hol (\overline{\DD})$, and $\varphi (z) = a \overline{z} + p(z)$. Suppose that
		\begin{align*}
			f = \sum_{j = 0}^ {\infty} {a_j e_j}
		\end{align*}
		is an eigenvector of $T_{\varphi}$ corresponding to the eigenvalue $\lambda$. Then $\{ a_j \}_{j = 0}^ {\infty}$ satisfies the recurrence relation
		\begin{align*}
			\sqrt{\frac{j + 1}{j + 2}} a a_{j + 1} + \sum_{k = 0}^ j {\sqrt{\frac{j - k + 1}{j + 1}} c_k a_{j - k}} = \lambda a_j, \qquad j \in \NN.
		\end{align*}
	\end{lem}
	
	\begin{proof}
		For each $n \in \NN$, let
		\begin{align*}
			p_n (z) = \sum_{k = 0}^ n {c_k z^ k} \qquad \textrm{and} \qquad \varphi_n (z) = a \overline{z} + p_n (z).
		\end{align*}
		Since $p$ is analytic on $\overline{\DD}$, $\{ p_n \}_{n = 0}^ {\infty}$ converges to $p$ uniformly on $\overline{\DD}$. It follows that
		\begin{align*}
			\| T_{\varphi} - T_{\varphi_n} \| = \| T_{\varphi - \varphi_n} \| = \| T_{p - p_n} \| \leqslant \sup_{z \in \DD} {|p (z) - p_n (z)|} \rightarrow 0 \ \ \ \ (n \rightarrow \infty),
		\end{align*}
		which gives
		\begin{align*}
			T_{\varphi} = a T_{\overline{z}} + \sum_{k = 0}^ {\infty} {c_k T_{z^ k}},
		\end{align*}
		where the series on the right-hand side converges in the operator norm. As a consequence, we have by Lemma \ref{lemma: Toeplitz operators with monomial symbols} that
		\begin{align*}
			T_{\varphi} f = a \sum_{j = 0}^ {\infty} {\sqrt{\frac{j + 1}{j + 2}} a_{j + 1} e_j} + \sum_{k = 0}^ {\infty} {c_k \sum_{j = k}^ {\infty} {\sqrt{\frac{j - k + 1}{j + 1}} a_{j - k} e_j}}.
		\end{align*}
		For each $j_0 \in \NN$,
		\begin{align*}
			\langle T_{\varphi} f, e_{j_0} \rangle_{L_a^ 2} & = a \sum_{j = 0}^ {\infty} {\sqrt{\frac{j + 1}{j + 2}} a_{j + 1} \langle e_j, e_{j_0} \rangle_{L_a^ 2}} + \sum_{k = 0}^ {\infty} {c_k \sum_{j = k}^ {\infty} {\sqrt{\frac{j - k + 1}{j + 1}} a_{j - k} \langle e_j, e_{j_0} \rangle_{L_a^ 2}}} \\
			& = \sqrt{\frac{j_0 + 1}{j_0 + 2}} a a_{j_0 + 1} + \sum_{k = 0}^ {j_0} {c_k \sqrt{\frac{j_0 - k + 1}{j_0 + 1}} a_{j_0 - k}},
		\end{align*}
		so we get that
		\begin{align*}
			T_{\varphi} f = \sum_{j = 0}^ {\infty} {\bigg[ \sqrt{\frac{j + 1}{j + 2}} a a_{j + 1} + \sum_{k = 0}^ j {\sqrt{\frac{j - k + 1}{j + 1}} c_k a_{j - k}} \bigg] e_j}.
		\end{align*}
		The desired recurrence relation follows by comparing coefficients of $T_{\varphi} f$ and $\lambda f$.
	\end{proof}
	
	For each $\lambda \in \sigma_{\mathrm{p}} (T_{\varphi})$, we let $f_{\lambda}$ be an eigenvector of $T_{\varphi}$ corresponding to $\lambda$ such that $f_{\lambda} (0) = 1$ (note that $a_0$ is nonzero, for otherwise we obtain by the recurrence relation in Lemma \ref{lemma: recurrence relation} that $a_0 = a_1 = \cdots = 0$). According to Lemma \ref{lemma: recurrence relation}, such an $f_{\lambda}$ is uniquely determined (thus the eigenspace of $\lambda$ is one-dimensional), so we have established a one-to-one correspondence, namely $\lambda \longmapsto f_{\lambda}$, from $\sigma_{\mathrm{p}} (T_{\varphi})$ into the set of eigenvectors of $T_{\varphi}$. The $j$-th coefficient $\langle f_{\lambda}, e_j \rangle_{L_a^ 2}$ of $f_{\lambda}$ will be denoted by $a_j (\lambda)$; by Lemma \ref{lemma: recurrence relation}, as a function of $\lambda$, $a_j (\lambda)$ is a polynomial of degree $j$. To apply Theorem \ref{lemma: S-G criterion} and Lemmas \ref{lemma: chaotic operators} and \ref{lemma: frequently hypercyclic operators}, we need to show that the subspace
	\begin{align*}
		\spa{\Bigg[ \bigcup_{\lambda \in E} {\ker{(T_{\varphi} - \lambda I)}} \Bigg]}
	\end{align*}
	is dense in $L_a^ 2$ for suitable $E$. In other words, we need to show that $g \in L_a^ 2$ is zero if
	\begin{align*}
		F (\lambda) :=\langle f_{\lambda}, g \rangle_{L_a^ 2}
	\end{align*}
	is zero for all $\lambda \in E$.
	
	\begin{lem} \label{lemma: absolute and uniform convergence}
		Let $a \in \CC \backslash \{ 0 \}$, $p \in \Hol (\overline{\DD})$, and $\varphi (z) = a \overline{z} + p(z)$. Suppose that $\widetilde{\varphi} (z) = a z^ {-1} + p (z)$ is univalent on $\overline{\DD}_0$. Let $\Omega_0$ be the interior of $\varphi (\TT)$. Then $\Omega_0 \subseteq \sigma_{\mathrm{p}} (T_{\varphi})$ and the series
		\begin{align}
			\sum_{j = 0}^ {\infty} {\overline{b}_j a_j (\lambda)}  \label{eq14}
		\end{align}
		converges to $F(\lambda)=\langle f_{\lambda}, g \rangle_{L_a^ 2}$ absolutely and uniformly on compact subsets of $\Omega_0$ for every $$g = \sum\limits_{j = 0}^ {\infty} {b_j e_j} \in L_a^ 2.$$ Consequently, $F$ is analytic on $\Omega_0$.
	\end{lem}
	
	\begin{proof}
		According to (a) of Lemma \ref{lemma: univalent}, the equation
		\begin{align*}
			\widetilde{\varphi} (z) = \lambda
		\end{align*}
		has no roots on $\DD_0$ for every $\lambda \in \Omega_0$. Hence, by (c) of Theorem \ref{theorem: spectrum and essential spectrum}, $\Omega_0$ is contained in the point spectrum of $T_{\varphi}$. The proof of the second assertion is analogous to that of \cite[Lemma 4.2]{LZ} and is therefore omitted.
	\end{proof}
	
	Part of the importance of Lemma \ref{lemma: absolute and uniform convergence} is that it tells us that under the assumption of this  lemma, if $F (\lambda)$ is zero on $E \subseteq \Omega_0$ then it is identically zero on $\Omega_0$, provided that $E$ has an accumulation point in $\Omega_0$. Thus, the problem we now face is how to deduce $g = 0$ from the condition $F (\lambda) = 0$ ($\lambda \in \Omega_0$). As shown in \cite{Aba}, if $p$ is an analytic linear polynomial, the recurrence relation in Lemma \ref{lemma: recurrence relation} is a three-term recurrence relation, and then by applying Favard's theorem (see \cite[Theorem 5.10]{Sch} if needed), there exists a finite positive Borel measure $\rd \nu$ such that $\{ a_j (\lambda) \}_{j = 0}^ {\infty}$ is orthogonal with respect to $\rd \nu$. It follows from  the orthogonality that $b_0 = b_1 = \cdots = 0$, and hence $g = 0$.

In what follows, by means of analytic functional calculus, we will obtain a sufficient condition for hypercyclicity and other dynamical properties of $T_{a \overline{z} + p}$ for any $p \in \Hol (\overline{\DD})$, which is close to (b) of Theorem \ref{theorem: hypercyclic Hardy-Toeplitz operators} in Section \ref{S1}.
	
	\begin{lem} \label{lemma: key lemma}
		Adopt the hypotheses and notation of Lemma \ref{lemma: absolute and uniform convergence}. Suppose further that $\widetilde{\varphi}$ is univalent on $R \DD_0$ for some $R > 1$ and $T_{\varphi}$ is $S$-spectrum connected. If the series in \textup{(\ref{eq14})} vanishes identically on $\Omega_0$, then we have $g = 0$.
	\end{lem}
	
	\begin{proof}
		Let
		\begin{align*}
			\psi (z) = p (\rho^ {-1} \overline{z}) + a \rho z,
		\end{align*}
		where $R^ {-1} < \rho < 1$. If we write $p (z) = \sum\limits_{k = 0}^ {\infty} {c_k z^ k}$, then direct calculations analogous to those in the proof of Lemma \ref{lemma: recurrence relation} give
		\begin{align*}
			T_{\psi} e_j = \sqrt{\frac{j + 1}{j + 2}} a \rho e_{j + 1} + \sum_{k = 0}^ {j} {\sqrt{\frac{j - k + 1}{j + 1}} c_k \rho^ {-k} e_{j - k}}, \qquad j \in \NN.
		\end{align*}
		Thus, if we put
		\begin{align*}
			v_j = \rho^ j e_j, \qquad j \in \NN,
		\end{align*}
		then
		\begin{align}
			T_{\psi} v_j = \sqrt{\frac{j + 1}{j + 2}} a v_{j + 1} + \sum_{k = 0}^ j {\sqrt{\frac{j - k + 1}{j + 1}} c_k v_{j - k}}, \qquad j \in \NN. \label{eq17}
		\end{align}
		
		Next, we show that $\sigma (T_{\psi}) \subseteq \Omega_0$ whenever $\rho$ is close to $1$. According to Lemma \ref{lemma: property of J}, it is sufficient to show that $\sigma (T_{J \psi}) \subseteq \Omega_0$ for $\rho$ close to $1$, where $J$ is the mapping $f (z) \longmapsto f (\overline{z})$. The analytic extension of $(J \psi) |_{\TT}$ is given by
		\begin{align*}
			\widetilde{J \psi} (z) = \frac{a \rho}{z} + \sum_{k = 0}^ {\infty} {c_k \Big( \frac{z}{\rho} \Big)^ k} = \widetilde{\varphi} \Big( \frac{z}{\rho} \Big).
		\end{align*}
		Since $\widetilde{\varphi}$ is univalent on $R \DD_0$ and $1 < \rho^ {-1} < R$, we find that $\widetilde{J \psi}$ is univalent on $\overline{\DD}_0$. If we let $\Omega_0'$ and $\Omega_{\infty}'$ be the interior and exterior of $(J \psi) (\TT)$, respectively,  then we have by (c) of Lemma \ref{lemma: univalent} that $$\Lambda_{\mathrm{c}} (T_{J \psi}) = (J \psi) (\TT) \sqcup \Omega_0'.$$ It is easy to see that
		\begin{align*}
			D_{J \psi} = \left\{ \frac{z^ 2 p' (z)}{a} : z \in \rho^ {-1} \DD_0 \right\}.
		\end{align*}
		Since $T_{\varphi}$ is $S$-spectrum connected, there is a $\delta > 0$ such that
		\begin{align*}
			\delta \leqslant \dist{(D_{\varphi}, W)} = \inf{\left\{ \left| \frac{z^ 2 p' (z)}{a} - \frac{m + 2}{m + 1} \right| : z \in \DD_0 \ \textrm{ and } \ m \in \NN \right\}},
		\end{align*}
where $W$ is the set defined in (\ref{W}).  Hence, by choosing $\rho$ sufficiently close to $1$, we get that
		\begin{align*}
			\frac{\delta}{2} \leqslant \dist{(D_{J \psi}, W)} = \inf{\left\{ \left| \frac{z^ 2 p' (z)}{a} - \frac{m + 2}{m + 1} \right| : z \in \rho^ {-1} \DD_0  \ \textrm{ and } \ m \in \NN \right\}}.
		\end{align*}
	Then it follows from  Lemma \ref{lemma: connectedness for univalent case} that the spectrum of $T_{J \psi}$ is connected.  Therefore,
		\begin{align*}
			\sigma (T_{J \psi}) = \Lambda_{\mathrm{c}} (T_{J \psi}) = (J \psi) (\TT) \sqcup \Omega_0',
		\end{align*}
		which is contained in $\Omega_0$.
		
		Fixing a $\rho \in (R^ {-1}, 1)$ close to $1$, we now prove that $g = 0$. Since the series defined in (\ref{eq14}) converges uniformly on compact subsets of $\Omega_0$ (by Lemma \ref{lemma: absolute and uniform convergence}) and  $\sigma (T_{\psi}) \subseteq \Omega_0$, we obtain
		\begin{align*}
			\sum_{j = 0}^ {\infty} {\overline{b}_j a_j (T_{\psi})} = 0,
		\end{align*}
		where the series on the left-hand side converges in the operator norm, see \cite[Theorem 10.27]{Rud}. Consequently,
		\begin{align}
			\sum_{j = 0}^ {\infty} {\overline{b}_j a_j (T_{\psi}) v_0} = 0. \label{eq18}
		\end{align}

		We claim that
		\begin{align}
			v_j = a_j (T_{\psi}) v_0, \qquad j \in \NN. \label{eq15}
		\end{align}
		Let us show the claim  above by induction. Since $a_0 = 1$, (\ref{eq15}) holds for $j = 0$. Suppose that (\ref{eq15}) holds for $0 \leqslant j \leqslant N$. We are going to show that (\ref{eq15}) is also true for $j = N + 1$. Indeed, invoking the recurrence formula in Lemma \ref{lemma: recurrence relation}, we find that
		\begin{align*}
			\sqrt{\frac{N + 1}{N + 2}} a a_{N + 1} (T_{\psi}) + \sum_{k = 0}^ {N} {\sqrt{\frac{N - k + 1}{N + 1}} c_k a_{N - k} (T_{\psi})} = T_{\psi} a_N (T_{\psi}),
		\end{align*}
		and hence
		\begin{align}
			\sqrt{\frac{N + 1}{N + 2}} a a_{N + 1} (T_{\psi}) v_0 + \sum_{k = 0}^ N {\sqrt{\frac{N - k + 1}{N + 1}} c_k v_{N - k}} = T_{\psi} v_N. \label{eq16}
		\end{align}
		Comparing (\ref{eq16}) with (\ref{eq17}), we arrive at
		\begin{align*}
			v_{N + 1} = a_{N + 1} (T_{\psi}) v_0,
		\end{align*}
		as required. Therefore, (\ref{eq18}) gives that
		\begin{align*}
			\sum_{j = 0}^ {\infty} {\overline{b}_j v_j} = 0.
		\end{align*}
		Since $\{ v_j \}_{j = 0}^ {\infty}$ is an orthogonal basis of $L_a^ 2$, we conclude that $b_0 = b_1 = \cdots = 0$, and hence $g = 0$.
	\end{proof}
	
	We now present the main theorem of this section.
	
	\begin{thm} \label{theorem: sufficient condition for hypercyclicity}
		Let $a \in \CC \backslash \{ 0 \}$, $p \in \Hol (\overline{\DD})$, and $\varphi (z) = a \overline{z} + p (z)$. Then $T_{\varphi}$ is hypercyclic, weakly mixing, mixing, chaotic and frequently hypercyclic on the Bergman space if
		\begin{enumerate}[label = \textup{(\alph*)}]
			\item $\widetilde{\varphi} (z) = a z^ {-1} + p (z)$ is univalent on $R \DD_0$ for some $R > 1$,
			\item $\DD \cap \sigma (T_{\varphi}) \neq \varnothing$ and $\widehat{\DD} \cap \sigma (T_{\varphi}) \neq \varnothing$, and
			\item $T_{\varphi}$ is $S$-spectrum connected, i.e.,
			\begin{align*}
				\inf{\left\{ \left| \frac{z^ 2 p' (z)}{a} - \frac{m + 2}{m + 1} \right| : z \in \DD_0, m \in \NN \right\}} > 0.
			\end{align*}
		\end{enumerate}
	\end{thm}

	\begin{proof}
		According to Lemma \ref{lemma: dynamical properties}, it is enough to show that $T_{\varphi}$ is mixing, chaotic and frequently hypercyclic if Conditions (a)--(c) hold.  Since the proofs of these three dynamical properties are quite similar, we give only the proof of the frequent hypercyclicity. First, by using Condition (c) and Lemmas \ref{lemma: connectedness}, \ref{lemma: univalent} and \ref{lemma: connectedness for univalent case}, we obtain
		\begin{align*}
			\sigma (T_{\varphi}) = \Lambda_c (T_{\varphi}) = \varphi (\TT) \sqcup \Omega_0,
		\end{align*}
		where $\Omega_0$ is the interior of $\varphi (\TT)$. Thus, Condition (b) implies that $E := \Omega_0 \cap \TT$ is nonempty and open in $\TT$.
		
		Recall that $\rd m$ is the normalized Lebesgue measure on $\TT$. We show that $T_{\varphi}$ has a perfectly $m$-spanning set of $\TT$-eigenvectors, and hence $T_{\varphi}$ is frequently hypercyclic by Lemma \ref{lemma: frequently hypercyclic operators}. Let $A$ be any Borel subset of $\TT$ with $m (A) = m (\TT) = 1$. If $g$ is a function in $L_a^ 2$ that is orthogonal to
		\begin{align*}
			\mathcal{D}_{A \cap E} := \spa{\Bigg[ \bigcup_{\lambda \in A \cap E} {\ker{(T_{\varphi} - \lambda I)}} \Bigg]},
		\end{align*}
		then we get that
		\begin{align*}
			F (\lambda) = \langle f_{\lambda}, g \rangle_{L_a^ 2} = 0, \qquad \lambda \in A \cap E.
		\end{align*}
		Since $m (A) = 1$ and $E \subseteq \Omega_0$ is open in $\TT$, $A \cap E$ has accumulation points in $\Omega_0$, and thus by Lemmas \ref{lemma: absolute and uniform convergence} and \ref{lemma: key lemma} we conclude that $g = 0$. Therefore, $\mathcal{D}_{A \cap E}$ is dense in $L_a^ 2$. By the definition, $T_{\varphi}$ has a perfectly $m$-spanning set of $\TT$-eigenvectors.
\end{proof}

		\begin{remark}
			On the Hardy space $H^ 2$, by arguing as in \cite{BL} and then applying Theorem \ref{lemma: S-G criterion} and Lemmas \ref{lemma: chaotic operators} and \ref{lemma: frequently hypercyclic operators}, one can easily show that conditions in part (b) of Theorem \ref{theorem: hypercyclic Hardy-Toeplitz operators} also imply that $\mathbf{T}_{\varphi} = \mathbf{T}_{a \overline{z} + p}$ is mixing, weakly mixing, chaotic and frequently hypercyclic.
		\end{remark}

	According to \cite[Proposition 4.5]{GZ} and \cite[Theorem 3.1]{GZZ}, the spectrum of $T_{a \overline{z} + p}$ is always connected if $p$ is an analytic polynomial with $\deg{(p)} \leqslant 2$ or has the form $p(z) = c_0 + c_n z^ n$. This result indeed suggests that in these cases, Condition (c) in Theorem \ref{theorem: sufficient condition for hypercyclicity} automatically holds.
	
	\begin{coro} \label{corollary: sufficient condition for hypercyclicity for deg <= 2}
		Suppose that $a \in \CC \backslash \{ 0 \}$, $p$ is an analytic polynomial of degree less than or equal to $2$ or has the form $p(z) = c_0 + c_n z^ n$ with $c_n \neq 0$, and $\varphi (z) = a \overline{z} + p (z)$. Then $T_{\varphi}$ is hypercyclic, weakly mixing, mixing, chaotic, and frequently hypercyclic if
		\begin{enumerate}[label = \textup{(\alph*)}]
			\item $\widetilde{\varphi} (z) = a z^ {-1} + p (z)$ is univalent on $R \DD_0$ for some $R > 1$, and
			\item $\DD \cap \sigma (T_{\varphi}) \neq \varnothing$ and $\widehat{\DD} \cap \sigma (T_{\varphi}) \neq \varnothing$.
		\end{enumerate}
	\end{coro}

	\begin{proof}
		By Theorem \ref{theorem: sufficient condition for hypercyclicity}, it suffices to show that $T_{\varphi}$ is $S$-spectrum connected. Let $R_0$ be any number lying in $(1, R)$. We consider the symbol
		\begin{align*}
			\psi (z) = a R_0^ {-1} \overline{z} + p (R_0 z).
		\end{align*}
		By \cite[Proposition 4.5]{GZ} and \cite[Theorem 3.1]{GZZ}, $\sigma (T_{\psi})$ is connected. Moreover, $\widetilde{\psi}$ is univalent on $\DD_0$ since $R_0 \DD_0 \subseteq R \DD_0$. Hence, we obtain by Lemma \ref{lemma: connectedness for univalent case} that
		\begin{align*}
			D_{\psi} \cap W = \varnothing,
		\end{align*}
		where
		\begin{align*}
			D_{\psi} = \left\{  \frac{z^ 2 p' (z)}{a} : z \in R_0 \DD_0 \right\}.
		\end{align*}
		Consequently,
		\begin{align*}
			\left\{ \frac{z^ 2 p' (z)}{a} : z \in R_0 \DD \right\} \cap W = (D_{\psi} \cup \{ 0 \}) \cap W = \varnothing.
		\end{align*}
		Observing that $D_{\psi} \cup \{ 0 \}$ contains
		\begin{align*}
			\overline{D_{\varphi}} = \left\{ \frac{z^ 2 p' (z)}{a} : z \in \overline{\DD} \right\},
		\end{align*}
		we therefore conclude that there exists an open set $D_{\psi} \cup \{ 0 \}$ containing the compact set $\overline{D_{\varphi}}$ that does not intersect $W$, which easily yields that
		\begin{align*}
			\dist{(D_{\varphi}, W)} = \dist{(\overline{D_{\varphi}}, W)} \geqslant \delta
		\end{align*}
		for some $\delta > 0$. By the definition, $T_{\varphi}$ is $S$-spectrum connected.
	\end{proof}

	In particular, in the case that $p$ is a linear analytic polynomial, we obtain the following result as well as an alternative proof of Theorem \ref{theorem: Abad's result}.

	\begin{coro} \label{corollary: sufficient and necessary condition for dynamical properties when deg p = 1}
		Let $a, c_1 \in \CC \backslash \{ 0 \}$, $c_0 \in \CC$, and $\varphi (z) = a \overline{z} + c_0 + c_1 z$. Then the following are equivalent:
		\begin{enumerate}[label = \textup{(\alph*)}]
			\item $|a| > |c_1|$ and $\varphi (\DD) \cap \TT \neq \varnothing$.
			\item $T_{\varphi}$ is hypercyclic.
			\item $T_{\varphi}$ is weakly mixing.
			\item $T_{\varphi}$ is mixing.
			\item $T_{\varphi}$ is chaotic.
			\item $T_{\varphi}$ is frequently hypercyclic.
		\end{enumerate}
	\end{coro}

	\begin{proof}
		The implication (b) $\Rightarrow$ (a) was shown in \cite[Proposition 3.5]{LZ}. To complete the proof, according to Lemma \ref{lemma: dynamical properties} and Corollary \ref{corollary: sufficient condition for hypercyclicity for deg <= 2}, it suffices to show that $T_{\varphi}$ satisfies Conditions (a) and (b) of Corollary \ref{corollary: sufficient condition for hypercyclicity for deg <= 2} if $|a| > |c_1|$ and $\varphi (\DD) \cap \TT \neq \varnothing$.
		
		Letting $R$ be any positive number such that $1 < R^ 2 < \big|\frac{a}{c_1}\big|$, we show that $\widetilde{\varphi}$ is univalent on $R \DD_0$, thereby establishing Condition (a) of Corollary \ref{corollary: sufficient condition for hypercyclicity for deg <= 2}. Otherwise, there would exist distinct points $z_1, z_2 \in R \DD_0$ such that
		\begin{align*}
			\frac{a}{z_1} + c_1 z_1 = \frac{a}{z_2} + c_1 z_2.
		\end{align*}
		It follows that
		\begin{align*}
			\frac{a}{c_1} = z_1 z_2,
		\end{align*}
		which is a contradiction since the modulus of the right-hand side is less than $\big|\frac{a}{c_1}\big|$. To get Condition (b) of Corollary \ref{corollary: sufficient condition for hypercyclicity for deg <= 2}, first note that $\varphi (\DD)$ is an open elliptic disk which is the topological interior of $\sigma (T_{\varphi})$ (see \cite[Remark 2.4]{LZ}), and thus $\varphi (\DD) \cap \TT \neq \varnothing$ implies that
		\begin{align*}
			\DD \cap \sigma (T_{\varphi}) \supseteq \DD \cap \varphi (\DD) \neq \varnothing \qquad \textrm{and} \qquad \widehat{\DD} \cap \sigma (T_{\varphi}) \supseteq \widehat{\DD} \cap \varphi (\DD) \neq \varnothing.
		\end{align*}
		The proof is complete.
	\end{proof}

	\section{The Discrepancy in hypercyclicity of Toeplitz operators on the Hardy and Bergman spaces} \label{Discrepancy in Hypercyclicity Between Hardy and Bergman Spaces}
	
	As mentioned in Section \ref{S1}, in \cite[Theorem 4.1]{GZZ}, Guo, Zhao and Zheng constructed a harmonic polynomial $\varphi (z) = a \overline{z} + p (z)$ such that the spectrum of the Bergman--Toeplitz operator $T_{\varphi}$ is disconnected. This also yields a bounded harmonic symbol for which the spectra of the Toeplitz operators on the Hardy and Bergman spaces differ, since the spectra of Hardy--Toeplitz operators are always connected by Widom's theorem \cite[Corollary 7.46]{Dou}. In this section, we shall construct examples to illustrate that this spectral discrepancy leads to a difference in hypercyclicity.

The  main result of this section is as follows.
	
	\begin{thm} \label{theorem: an example}
		There exists a function $\Phi$ harmonic on $\DD$ and continuous on $\overline{\DD}$ such that the corresponding Hardy--Toeplitz operator $\mathbf{T}_{\Phi}$ is hypercyclic while the Bergman--Toeplitz operator $T_{\Phi}$ is non-hypercyclic.
	\end{thm}

	In this section, we use $\log{z}$ to denote the principal branch of the complex logarithm which is defined and analytic on $\CC \backslash (-\infty, 0]$, and maps $1$ to $0$. For $a \in \CC \backslash (-\infty, 0]$ and $b \in \CC$, we set
	\begin{align*}
		a^ b = \ee^ {b \log{a}}.
	\end{align*}
	Our construction begins with the function
	\begin{align*}
		s (z) = z (1 - z)^ {-2 \ee^ {\ii \alpha} \cos{\alpha}} = z \exp{\big[ -2 \ee^ {\ii \alpha} \cos{\alpha} \cdot \log{(1 - z)} \big]}, \qquad z \in \DD,
	\end{align*}
	where $-\frac{\pi}{2} < \alpha <\frac{\pi}{2}$. It is a so-called $\alpha$-spirallike function which is, by the definition, univalent on $\DD$, see \cite[Section 2.7]{Dur}. A direct calculation gives
	\begin{align*}
		s' (z) = (1 - z)^ {-2 \ee^ {\ii \alpha} \cos{\alpha}} \bigg( 1 + \frac{2 \ee^ {\ii \alpha} z\cos{\alpha}}{1 - z} \bigg).
	\end{align*}
	Hence,
	\begin{align*}
		F (z) & := \frac{z^ 2 s' (z)}{[s (z)]^ 2} = (1 - z)^ {2 \ee^ {\ii \alpha} \cos{\alpha}} \bigg( 1 + \frac{2 \ee^ {\ii \alpha} z\cos{\alpha}}{1 - z} \bigg) \\
		& = (1 - z)^ {2 \ee^ {\ii \alpha} \cos{\alpha} - 1} \big[ 1 + (2 \ee^ {\ii \alpha} \cos{\alpha} - 1) z \big] \\
		& = (1 - z)^ {\ee^ {2 \ii \alpha}} \big( 1 + \ee^ {2 \ii \alpha} z \big), \qquad z \in \DD.
	\end{align*}
	The role of $F$ will be clear later. At this stage, we show that $F$ has an important property that it maps $\DD$ onto a punctured neighborhood of the origin.

For convenience of exposition, we first introduce some notation. First, we make a change of variables $\xi = 1 - z$ to obtain
	\begin{align*}
		G (\xi) := F (1 - \xi) = \xi^ {\ee^ {2 \ii \alpha}} \big[ 1 + \ee^ {2 \ii \alpha} (1 - \xi) \big], \qquad \xi \in D (1, 1).
	\end{align*}
	For $0 < \beta <\frac{\pi}{2}$ and $R > 0$, let
	\begin{align*}
		S_{\beta, R} = \left\{ r \ee^ {\ii \theta} : -\beta < \theta < \beta \ \textrm{ and } \ 0 < r < R \right\},
	\end{align*}
	which is a sector. For $a \in \RR$ and $d > 0$, let
	\begin{align*}
		\Sigma_{d, a} = \left\{ x + \ii y : -d < y < d \  \textrm{ and } -\infty < x < a \right\},
	\end{align*}
	which is a half-strip. Observe that $\log{z}$ maps $S_{\beta, R}$ onto $\Sigma_{\beta, \log{R}}$ conformally.
	
	\begin{lem} \label{lemma: property of F}
		Fix an $\alpha \in (0, \frac{\pi}{4})$ satisfying $0 < \cos{(2 \alpha)} < \frac{1}{4}$. Let $\beta \in (0, \frac{\pi}{2})$ and $R > 0$, and suppose further that
		\begin{align*}
			\cos{(2 \alpha)} < \frac{\beta}{2 \ppi}, \qquad S_{\beta, R} \subseteq D (1, 1) \qquad \textrm{and} \qquad R < |1 + \ee^ {2 \ii \alpha}| (1 - \ee^ {-\frac{\beta}{ 2}}).
		\end{align*}
		Then there exists a $\delta = \delta (\beta, R) > 0$ such that
		\begin{align*}
			D (0, \delta) \backslash \{ 0 \} \subseteq G (S_{\beta, R}).
		\end{align*}
	\end{lem}
	
	\begin{proof}
		We first rewrite $G$ as
		\begin{align*}
			G (\xi) = (1 + \ee^ {2 \ii \alpha}) \xi^ {\ee^ {2 \ii \alpha}} \bigg( 1 - \frac{\ee^ {2 \ii \alpha}}{1 + \ee^ {2 \ii \alpha}} \xi \bigg) =: (1 + \ee^ {2 \ii \alpha}) H (\xi).
		\end{align*}
		It suffices to show that
		\begin{align*}
			D (0, \delta) \backslash \{ 0 \} \subseteq H (S_{\beta, R})
		\end{align*}
		for some $\delta > 0$. Let
		\begin{align*}
			L (\xi) = \ee^ {2 \ii \alpha} \log{\xi} + \log{\bigg( 1 - \frac{\ee^ {2 \ii \alpha}}{1 + \ee^ {2 \ii \alpha}} \xi \bigg)}, \qquad \xi \in S_{\beta, R}.
		\end{align*}
		Since
		\begin{align*}
			\left| \frac{\ee^ {2 \ii \alpha}}{1 + \ee^ {2 \ii \alpha}} \xi \right| & = \frac{|\xi|}{|1 + \ee^ {2 \ii \alpha}|} < \frac{R}{|1 + \ee^ {2 \ii \alpha}|} < 1 - \ee^ {-\frac{\beta}{2}} < 1, \qquad \xi \in S_{\beta, R},
		\end{align*}
		$L$ is well-defined on $S_{\beta, R}$, and hence $H = \ee^ L$ on $S_{\beta, R}$.
		
		First, note that $\ee^ {2 \ii \alpha} \log{\xi}$ maps $S_{\beta, R}$ onto $\ee^ {2 \ii \alpha} \Sigma_{\beta, \log{R}}$. We next show that $L$ maps $S_{\beta, R}$ onto an open set containing $$\Sigma' := \ee^ {2 \ii \alpha} \Sigma_{\frac{\beta}{2}, - \frac{\beta}{2}+\log{R} }.$$ For each $\lambda \in \Sigma'$, there exists a rectangle
		\begin{align*}
			\mathcal{R} := \ee^ {2 \ii \alpha} \left\{ x + \ii y : -\beta < y < \beta \ \textrm{ and } \ M < x < \log{R} \right\} \qquad (-\infty < M < \log{R})
		\end{align*}
		such that
		\begin{align}
			\dist{(\lambda, \partial \mathcal{R})} \geqslant \frac{\beta}{2}. \label{ineq1}
		\end{align}
		If we let
		\begin{align*}
			\mathcal{S} = \left\{ r \ee^ {\ii \theta} : -\beta < \theta < \beta \ \textrm{ and } \ \ee^ M < r < R \right\},
		\end{align*}
		then (\ref{ineq1}) gives that
		\begin{align*}
			\left| \ee^ {2 \ii \alpha} \log{\xi} - \lambda \right| \geqslant \frac{\beta}{2}, \qquad \xi \in \partial \mathcal{S}.
		\end{align*}
		On the other hand, for $\xi \in \partial \mathcal{S} \subseteq \overline{S_{\beta, R}}$ we also have
		\begin{align*}
			\left| \log{\bigg( 1 - \frac{\ee^ {2 \ii \alpha}}{1 + \ee^ {2 \ii \alpha}} \xi \bigg)} \right| & \leqslant \sum_{n = 1}^ {\infty} {\frac{1}{n} \left| \frac{\xi}{1 + \ee^ {2 \ii \alpha}} \right|^ n} = -\log{\bigg( 1 - \frac{|\xi|}{|1 + \ee^ {2 \ii \alpha}|} \bigg)} < \frac{\beta}{2},
		\end{align*}
		where the last inequality follows from the assumption that $$|\xi| \leqslant R < |1 + \ee^ {2 \ii \alpha}| (1 - \ee^ {-\frac{\beta} {2}}).$$  Therefore, we deduce that
		\begin{align*}
			\left| \log{\bigg( 1 - \frac{\ee^ {2 \ii \alpha}}{1 + \ee^ {2 \ii \alpha}} \xi \bigg)} \right| < \left| \ee^ {2 \ii \alpha} \log{\xi} - \lambda \right|, \qquad \xi \in \partial \mathcal{S}.
		\end{align*}
		It follows from Rouch\'e's theorem that $L-\lambda$ has a zero in $\mathcal{S}$, since $\ee^ {2 \ii \alpha} \log{\xi} - \lambda$ has a zero in $\mathcal{S}$. Because here $\lambda$ is an arbitrary point in $\Sigma'$ and $\mathcal{S}$ is always contained in $S_{\beta, R}$, we find that $L (S_{\beta, R})$ contains $\Sigma'$.
		
		For every sufficiently negative real number $c$ (i.e., $c$ is a negative real number whose absolute value is sufficiently large), basic geometry tells us that the intersection of the line $\re{z} = c$ and $\Sigma'$ is a vertical interval whose length is $\frac{\beta}{\cos{(2 \alpha)}} > 2 \ppi$, and thus the exponential function maps this interval onto the circle $|z| = \ee^ c$. As a result, $H = \ee^ L$ maps $S_{\beta, R}$ onto an open set which contains $D (0, \delta) \backslash \{ 0 \}$ for some $\delta > 0$. The proof of the lemma is complete.
	\end{proof}

	Henceforth, we fix an $\alpha \in (0, \frac{\pi}{4})$ satisfying $0 < \cos(2 \alpha) < \frac{1}{4}$. It then follows from Lemma \ref{lemma: property of F} that there exist a $z_0 \in \DD$ and an $m_0 \in \NN$ such that
	\begin{align*}
		F (z_0) = \frac{z_0^ 2 s' (z_0)}{[s (z_0)]^ 2} = -\frac{1}{m_0 + 1}.
	\end{align*}
	Let $\rho$ be a number such that $|z_0| < \rho < 1$, and let
	\begin{align*}
		g (z) = \frac{s (\rho z)}{\rho}, \qquad z \in \rho^ {-1} \DD.
	\end{align*}
	Some elementary properties of $g$ are collected in the following lemma.
	
	\begin{lem} \label{lemma: properties of g}
		For the function $g$ defined as above, the following are true:
		\begin{enumerate}[label = \textup{(\alph*)}]
			\item $g$ is analytic and univalent on $\overline{\DD}$.
			\item $g(0) = 0$ and $g' (0) = 1$.
			\item If $w_0 = \rho^ {-1} z_0 \in \DD$, then
			\begin{align*}
				\frac{w_0^ 2 g' (w_0)}{[g (w_0)]^ 2} = -\frac{1}{m_0 + 1}.
			\end{align*}
		\end{enumerate}
	\end{lem}
	
	\begin{proof}
		The above  three assertions follow directly from the definition of $g$.
	\end{proof}

	Now, we let
	\begin{align*}
		p (z) = \frac{1}{g(z)} - \frac{1}{z} = \frac{z - g(z)}{z g(z)}.
	\end{align*}
	By (a) and (b) of Lemma \ref{lemma: properties of g}, we can write $g(z) = z h(z)$, where $h$ is analytic on $\overline{\DD}$ with $h(0) = 1$. Thus, we have
	\begin{align*}
		p(z) = \frac{1 - h(z)}{zh(z)}.
	\end{align*}
	Using (a) of Lemma \ref{lemma: properties of g} again, $0$ is the only zero of $g$ on $\overline{\DD}$, implying that $h$ is zero-free on $\overline{\DD}$. Therefore, $p$ is analytic on $\overline{\DD}$. Lastly, we let
	\begin{align}
		\varphi (z) = \overline{z} + p(z) \label{eq2}
	\end{align}
	and $\lambda_0 = \widetilde{\varphi} (w_0) = \widetilde{\varphi} (\rho^ {-1} z_0)$.

	\begin{lem} \label{lemma: no hypercyclic scalars}
		Let $\varphi$ and $\lambda_0$ be defined as above. Then $\lambda_0$ is an isolated point of the spectrum of the Bergman--Toeplitz operator $T_{\varphi}$. As a consequence of Theorem \ref{theorem: necessary condition for hypercyclicity}, $T_{k \varphi} = k T_{\varphi}$ is non-hypercyclic on the Bergman space $L_a^2$ for every $k \in \CC$.
	\end{lem}

	\begin{proof}
		According to Lemma \ref{lemma: properties of g},
		\begin{align*}
			\widetilde{\varphi} (z) = \frac{1}{z} + p (z) = \frac{1}{g (z)}
		\end{align*}
		is univalent on $\overline{\DD}_0$. The univalence of $\widetilde{\varphi}$ on $\overline{\DD}_0$ implies that $z = w_0$ is the only root of the equation
		\begin{align*}
			\widetilde{\varphi} (z) = \lambda_0.
		\end{align*}
		Since
		\begin{align*}
			w_0^ 2 p' (w_0) = 1 - \frac{w_0^ 2 g' (w_0)}{g (w_0)^ 2} = \frac{m_0 + 2}{m_0 + 1},
		\end{align*}
		we conclude by Theorem \ref{theorem: spectrum and essential spectrum} that $\lambda_0$ is an eigenvalue of $T_{\varphi}$. According to Lemma \ref{lemma: univalent}, $\lambda_0 \in \widetilde{\varphi} (\DD_0)$ is in the exterior of $\varphi (\TT)$, and hence
		\begin{align*}
			\lambda_0 \in \Lambda_{\mathrm{d}} (T_{\varphi}) =\big \{ \lambda \in \CC \backslash \varphi (\TT) : \wind{(\varphi (\TT), \lambda)} = 0 \ \textrm{ and } \ \ker{(T_{\varphi} - \lambda I)} \neq \{ 0  \} \big\},
		\end{align*}
		which shows that $\lambda_0$ is an isolated point of $\sigma (T_{\varphi})$ by Lemma \ref{lemma: Lambdad is discrete}.
	\end{proof}

	However, our next lemma shows that for the symbol $\varphi$ given above, there always exists a $k_0 \in \CC$ such that $\mathbf{T}_{k_0 \varphi} = k_0 \mathbf{T}_{\varphi}$ is hypercyclic on the Hardy space $H^2$.
	
	\begin{lem} \label{lemma: hypercyclic scalar}
		Let $a \in \CC \backslash \{ 0 \}$, $q \in \Hol (\DD) \cap C (\overline{\DD})$, and $\psi (z) = \frac{a}{z} + q (z)$. Suppose that $\psi$ is univalent on $\overline{\DD}_0$. Then there exists a  number $k_0$ such that $\mathbf{T}_{k_0 \psi}$ is hypercyclic on the Hardy space $H^2$.
	\end{lem}

	\begin{proof}
		Let $\Omega_0$ and $\Omega_{\infty}$ be the interior and exterior of $\psi (\TT)$, respectively. According to Lemma \ref{lemma: univalent}, $\Omega_{\infty} = \widetilde{\psi} (\DD_0)$. Let $k_0$ be a complex number such that the open set $k_0 \Omega_0$ intersects $\TT$. As a consequence,
		\begin{align*}
			\DD \cap (k_0 \Omega_0) \neq \varnothing \qquad \textrm{and} \qquad \widehat{\DD} \cap (k_0 \Omega_0) \neq \varnothing.
		\end{align*}
		Observe that $k_0 \Omega_0$ and $k_0 \Omega_{\infty}$ are the interior and exterior of $(k_0 \psi) (\TT)$, respectively.  We obtain by Lemma \ref{lemma: univalent} that
		\begin{align*}
			\CC \backslash (k_0 \psi) (\DD_0) = \CC \backslash (k_0 \Omega_{\infty}) = (k_0 \Omega_0) \sqcup (k_0 \psi) (\TT) \supseteq k_0 \Omega_0.
		\end{align*}
		As a result,
		\begin{align*}
			\DD \cap [\CC \backslash (k_0 \psi) (\DD_0)] \neq \varnothing \qquad \textrm{and} \qquad \widehat{\DD} \cap [\CC \backslash (k_0 \psi) (\DD_0)] \neq \varnothing.
		\end{align*}
		It follows from Theorem \ref{theorem: hypercyclic Hardy-Toeplitz operators} that $\mathbf{T}_{k_0 \psi}$ is hypercyclic on $H^2$.
	\end{proof}

We are now ready to prove the main result of this section.

	\begin{proof}[Proof of Theorem \ref{theorem: an example}]
		We still let $\varphi$ be the function defined in (\ref{eq2}). Let $k_0$ be a complex number such that $\mathbf{T}_{k_0 \varphi}$ is hypercyclic on the Hardy space $H^2$, whose existence is guaranteed by Lemma \ref{lemma: hypercyclic scalar}. On the other hand, $T_{k_0 \varphi}$ is non-hypercyclic on the Bergman space $L_a^2$ by Lemma \ref{lemma: no hypercyclic scalars}. We can thus take $\Phi = k_0 \varphi$, to finish the proof.
\end{proof}

	In view of Theorem \ref{theorem: an example}, it is natural to ask the following question: 
\begin{ques}
Does there exist a function $\Psi$
 such that the Toeplitz operator with symbol $\Psi$
 is hypercyclic on the Bergman space but non-hypercyclic on the Hardy space?
  \end{ques}
 
 If the symbol is not required to be harmonic, such examples are easy to construct.
	
	\begin{prop}\label{NHHB}
		Let $$\Psi (z) = \overline{z} (2 - |z|^ 2), \ \ \ \  z\in \overline{\mathbb D}.$$ Then the Hardy--Toeplitz operator $\mathbf{T}_{\Psi}$ is non-hypercyclic while the Bergman--Toeplitz operator $T_{\Psi}$ is hypercyclic.
	\end{prop}

	\begin{proof}
		For $z \in \TT$, we have $\Psi (z) = \overline{z}$. Thus, $\mathbf{T}_{\Psi}$ is exactly the unweighted backward shift operator $\mathbf{T}_{\overline{z}}$. Since $\mathbf{T}_{\overline{z}}$ is a contraction, it is non-hypercyclic on the Hardy space $H^2$.

Recall that $e_j (z) = \sqrt{j+1}z^j\ (j\in \mathbb N)$, which is an orthonormal basis for the Bergman space $L_a^2$. Then we can calculate as in the proof of Lemma \ref{lemma: Toeplitz operators with monomial symbols} to obtain that
		\begin{align*}
			T_{\Psi} e_j (z) =
			\begin{cases}
				0, & j = 0, \vspace{2mm}\\
				\sqrt{\frac{j}{j + 1}} \frac{j + 3}{j + 2} e_{j - 1} (z), & j \in \NN_+.
			\end{cases}
		\end{align*}
		Note that $T_{\Psi}$ is a weighted backward shift, whose weight sequence $\{ w_j \}_{j = 1}^ {\infty}$ is given by
		\begin{align*}
			w_j = \sqrt{\frac{j}{j + 1}} \frac{j + 3}{j + 2}, \qquad j \in \NN_+.
		\end{align*}
		Since the sequence
		\begin{align*}
			v_n := w_1 w_2 \cdots w_n = \frac{n + 3}{3 \sqrt{n + 1}}\ \ \ \ (n=1, 2, \cdots)
		\end{align*}
		increases to $+\infty$ as $n \rightarrow \infty$, we conclude by \cite[Theorem 2.8]{Sal1} that $T_{\Psi}$ is hypercyclic on the Bergman space $L_a^2$.
	\end{proof}

\subsection*{Acknowledgment}
The first and third authors were supported by the National Natural Science Foundation of China (12371125) and the Chongqing Natural Science Foundation (CSTB2024NSCQ-MSX0177). The second author was supported by the Natural Science Starting Project of SWPU (2024QHZ026).


\begin{thebibliography}{99}
		\bibitem{Aba} O. Abad, On the dynamics of Toeplitz operators over Bergman spaces, \href{https://arxiv.org/abs/2609.22017}{arXiv:2609.22017}.
		
		\bibitem{ABCL} E. Abakumov, A. Baranov, S. Charpentier, A. Lishanskii, New classes of hypercyclic Toeplitz operators, \textit{Bull. Sci. Math.}, 2021, 168: article number 102971.
		
		\bibitem{BL} A. Baranov, A. Lishanskii, Hypercyclic Toeplitz operators, \textit{Results. Math.}, 2016, 70: 337--347.
		
		\bibitem{BM} F. Bayart, E. Matheron, \textit{Dynamics of Linear Operators}, Cambridge University Press, Cambridge, 2009.
		
		\bibitem{DE} \"O. De\u{g}er, B. B. Eski\c{s}ehirli, Disjoint hypercyclic Toeplitz operators, \textit{Arch. Math.}, 2025, 124: 301--310.
		
		\bibitem{Dou} R. Douglas, \textit{Banach Algebra Techniques in Operator Theory}, second edition, Springer, New York, 1998.
		
		\bibitem{Dur} P. L. Duren, \textit{Univalent Functions}, Springer-Verlag, New York, 1983.
		
		\bibitem{FGO} E. Fricain, S. Grivaux, M. Ostermann, Hypercyclicity of Toeplitz operators with smooth symbols, \href{https://arxiv.org/abs/2502.03303}{arXiv:2502.03303}.
		
		\bibitem{G-E} K.-G. Grosse-Erdmann, Hypercyclic and chaotic weighted shifts, \textit{Studia Math.}, 2000, 139: 47--68.
		
		\bibitem{GZ} N. Guan, X. Zhao, Invertibility of Bergman-Toeplitz operators with harmonic polynomial symbols, \textit{Sci. China Math.}, 2020, 63: 965--978.
		
		\bibitem{GLS} Z. Guo, L. Lin, Y. Shu, Frequent hypercyclicity and chaoticity of Toeplitz operators and their tensor products, \textit{Proc. Math. Sci.}, 2021, 131: article number 49.
		
		\bibitem{GS} G. Godefroy, J. Shapiro, Operators with dense, invariant cyclic vector manifolds, \textit{J. Funct. Anal.}, 1991, 98: 229--269.
		
		\bibitem{GZZ} K. Guo, X. Zhao, D. Zheng, The spectral picture of Bergman-Toeplitz operators with harmonic polynomial symbols, \textit{Ark. Mat.}, 2023, 61: 343--374.
		
		\bibitem{LZ} Q. Leng, X. Zhao, Hypercyclic Bergman--Toeplitz operators with harmonic polynomial symbols, \textit{Results. Math.}, 2026, 81: article number 167.
		
		\bibitem{Rol} S. Rolewicz, On orbits of elements, \textit{Studia Math.}, 1969, 32: 17--22.
		
		\bibitem{Rud} W. Rudin, \textit{Functional Analysis}, second edition, McGraw-Hill, New York, 1973.
		
		\bibitem{Sal1} H. N. Salas, Hypercyclic weighted shifts, \textit{Trans. Amer. Math. Soc.}, 1995, 347: 993--1004.
		
		\bibitem{Sal2} H. N. Salas, Supercyclicity and weighted shifts, \textit{Studia Math.}, 1999, 135: 55--74.


	\bibitem{Sch}
K. Schm\"udgen, \textit{The Moment Problem}, Springer, Cham, 2017.
		
		\bibitem{Shk} S. Shkarin, Orbits of coanalytic Toeplitz operators and weak hypercyclicity, \href{https://arxiv.org/abs/1210.3191}{arXiv:1210.3191}.
		
	



		\bibitem{Zhu} K. Zhu, \textit{Operator Theory in Function Spaces}, second edition, American Mathematical Society, 2007.
	\end{thebibliography}
\end{document}